\documentclass[a4paper]{amsart}
\usepackage{enumerate, amsfonts, amssymb, amsthm, mathrsfs, booktabs}
\usepackage{tabularx}

\usepackage{caption}
\usepackage[arrow, matrix, curve]{xy}
\usepackage{tikz}
\usepackage{graphicx}
\usetikzlibrary{patterns,shapes,decorations.pathmorphing,decorations.pathreplacing,calc,arrows}
\usepackage{blindtext, rotating}
\usepackage{pgfplots}
\usepackage{comment}
\pgfplotsset{compat=newest}
\usepackage[para]{footmisc}

\usepackage[colorlinks=true]{hyperref}
\numberwithin{equation}{section}
\usepackage{cleveref}

\newcommand{\FindStat}[1]{\url{www.findstat.org/#1}}
\newcommand{\OEIS}[1]{\url{www.oeis.org/#1}}

\numberwithin{equation}{section}
\theoremstyle{plain}
\newtheorem{lemma}[equation]{Lemma}
\newtheorem{theorem}[equation]{Theorem}
\newtheorem{conjecture}[equation]{Conjecture}
\newtheorem{corollary}[equation]{Corollary}
\newtheorem{proposition}[equation]{Proposition}
\newtheorem{problem}{Problem}

\theoremstyle{definition}
\newtheorem{definition}[equation]{Definition}
\newtheorem{remark}[equation]{Remark}

\newtheorem{example}[equation]{Example}

\crefname{proposition}{Prop.}{Props.}
\crefname{corollary}{Cor.}{Cors.}
\crefname{theorem}{Thm.}{Thms.}
\crefname{equation}{Eq.}{Eqs.}

\Crefname{proposition}{Proposition}{Propositions}
\Crefname{theorem}{Theorem}{Theorems}
\Crefname{corollary}{Corollary}{Corollaries}
\Crefname{equation}{Equation}{Equations}

\DeclareMathOperator{\Ext}{Ext}%
\DeclareMathOperator{\id}{id}%
\DeclareMathOperator{\gldim}{gldim}%

\definecolor{lightgrey}{rgb}{0.7,0.7,0.7}

\makeatletter
\long\def\ifnodedefined#1#2#3{%
	\@ifundefined{pgf@sh@ns@#1}{#3}{#2}%
}
\makeatother

\newcommand{\drawPath}[4]
{%
	\draw[rounded corners=1, line width=2, #3] #4
	\foreach \dir in {#1}%
	{
		\ifnum\dir=#2
		-- ++(1,0)
		\else
		-- ++(0,1)
		\fi
	};}

\title[Homological algebra of Nakayama algebras]{Homological algebra of Nakayama algebras and 321-avoiding permutations}

\date{\today}
\author{Eirini Chavli}
\address[E.~Chavli]{Institute for Discrete Structures and Symbolic Computation of the University of Stuttgart, Germany}
\email{eirini.chavli@mathematik.uni-stuttgart.de}
\author{Rene Marczinzik}
\address[R.~Marczinzik]{Mathematical Institute of the University of Bonn, Germany}
\email{marczinzik.rene@googlemail.com}


\begin{document}
	\maketitle
	
	\begin{abstract}
		Linear Nakayama algebras over a field $K$ are in natural bijection to Dyck paths and Dyck paths are in natural bijection to 321-avoiding permutations via the Billey-Jockusch-Stanley bijection. Thus to every 321-avoiding permutation $\pi$  we can associate in a natural way a linear Nakayama algebra $A_{\pi}$.
		We give a homological interpretation of the fixed points statistic of 321-avoiding permutations using Nakayama algebras with a linear quiver. We furthermore show that the space of self-extensions for the Jacobson radical of a linear Nakayama algebra $A_{\pi}$ is isomorphic to $K^{\mathfrak{s}(\pi)}$, where $\mathfrak{s}(\pi)$ is defined as the cardinality $k$ such that $\pi$ is the minimal product of transpositions of the form $s_i=(i,i+1)$ and $k$ is the number of distinct $s_i$ that appear. 
	\end{abstract}

	\section{Introduction}
	We assume all algebras are finite dimensional over a field $K$ and are given by a connected quiver and admissible relations.
	Nakayama algebras with a linear quiver and $n$ simple modules are in natural bijection to Dyck paths and Dyck paths are in natural bijection to 321-avoiding permutations via the Billey-Jockusch-Stanley bijection. Thus to every 321-avoiding permutation $\pi$ we can associate a Nakayama algebra with a linear quiver that we denote by $A_{\pi}$.
	In \cite{IM}, it was shown that for the incidence algebra of a finite distributive lattice $L$, the number of indecomposable projective $A$-modules with injective dimension one is equal to the number of join-irreducible elements of $L$ (we refer to Corollary \ref{injdimprojectives} for an explicit description of indecomposable projective modules with injective dimension one). Thus it is a natural question whether this number of indecomposable projective $A$-modules with injective dimension one also has a combinatorial interpretation for other finite dimensional algebras.
	An important statistic for 321-avoiding permutations is the number of fixed points, see for example \cite{HRS}.
	Our first main result gives a homological interpretation of fixed points of 321-avoiding permutations using Nakayama algebras and the number of indecomposable projective $A$-modules with injective dimension one.
	\begin{theorem}
		Let $A_{\pi}$ be a Nakayama algebra corresponding to the 321-avoiding permutation $\pi$.
		Then the number of indecomposable projective $A$-modules with injective dimension one is equal to the number of fixed points of $\pi$.
	\end{theorem}
	
	As a part of the proof of the above theorem, we provide a formula for the fixed points of every 321-avoiding permutation (see Corollary \ref{corollaryP}). Moreover, this proof also specifies which indecomposable projectives have injective dimension 1 (see proof of Proposition \ref{P}).

	A classical topic in homological algebra is the calculation of extension spaces between modules of a ring. One of the most important modules for finite dimensional algebras is the Jacobson radical $J$ that is defined as the intersection of all maximal right ideals. A central result is that the algebra is semi-simple if and only if the Jacobson radical is zero. 
	In this article we want to look at the vector space of extensions $\Ext_A^1(J,J)$ that classifies short exact sequences of the form $0 \rightarrow J \rightarrow W \rightarrow J \rightarrow 0$.
	Viewing the symmetric group $S_n$ as a Coxeter group with standard generators the transpositions $s_i=(i,i+1)$, the support size $\mathfrak{s}(\pi)$ of a permutation $\pi$ is defined as the cardinality $k$ such that $\pi$ is the minimal product of transpositions of the form $s_i$ and $k$ is the number of distinct $s_i$ that appear (see \url{http://www.findstat.org/StatisticsDatabase/St000019} for this statistic on permutations).  
	Our second main result relates the space of self-extensions of the Jacobson radical and the support size of a 321-avoiding permutation.
	\begin{theorem} \label{conjecture}
		Let $A_{\pi}$ be a linear Nakayama algebra with Jacobson radical $J$ associated to the 321-avoiding permutation $\pi$.
		Then $\Ext_{A_{\pi}}^1(J,J) \cong K^{\mathfrak{s}(\pi)}$.
	\end{theorem}

	As a corollary of the previous theorem we obtain that the number of linear Nakayama algebras $A$ with Jacobson radical $J$ having $n+2$ simple modules such that \linebreak $\dim(\Ext_A^1(J,J))=k$ is equal to the number of standard tableaux of shape $[n,k]$. In particular the maximal vector space dimension of $\Ext_A^1(J,J)$ is equal to $n$ for linear Nakayama algebras with $n+2$ simple modules and the number of algebras where this maximal vector space dimension is attained is given by the Catalan number $C_n$.

	\section{Preliminaries} 
	We always assume that algebras are finite dimensional, connected quiver algebras over a field $K$.
	A module $M$ is called \emph{uniserial} if it has a unique composition series. A \emph{Nakayama algebra} is by definition an algebra such that every indecomposable module is uniserial.
	They can be characterized as the quiver algebras having either a linear oriented line as a quiver or a linear oriented cycle.
	
	The quiver of a Nakayama algebra with a cycle as a quiver:
	$$\begin{xymatrix}{ &  \circ^0 \ar[r] & \circ^1 \ar[dr] &   \\
			\circ^{n-1} \ar[ur] &     &     & \circ^2 \ar[d] \\
			\circ^{n-2} \ar[u] &  &  & \circ^3 \ar[dl] \\
			& \circ^5 \ar @{} [ul] |{\ddots} & \circ^4 \ar[l] &  }\end{xymatrix}$$
	\newline
	\newline
	The quiver of a Nakayama algebra with a line as a quiver:
	$$\begin{xymatrix}{ \circ^0 \ar[r] & \circ^1 \ar[r] & \circ^2 \ar @{} [r] |{\cdots} & \circ^{n-2} \ar[r] & \circ^{n-1}}\end{xymatrix}$$
	
	In this article we will only work with Nakayama algebras having a line as a quiver. Thus Nakayama algebra means in the rest of the article a Nakayama algebra that is a quiver algebra with a connected line as a quiver.
	Let $e_i$ be the primitive idempotents corresponding to the points in the quiver of a Nakayama algebra $A$.
	A Nakayama algebra is uniquely determined by the dimensions of the indecomposable projective modules $e_i A$. Set $c_i := \dim( e_i A)$, then the list $[c_0,c_1,\dots,c_{n-1}]$ is called the \emph{Kupisch series} of $A$ when $A$ has $n$ simple modules.
	Every indecomposable module of a Nakayama algebra $A$ is of the form $e_i A/e_i J^k$ for some $i=0,1,\dots,n-1$ and $1 \leq k \leq c_i$.
	The indecomposable injective $A$-modules are the modules $D(Ae_i)$ and we denote their vector space dimensions by $d_i$.
	We have that $d_{i}=\min\big\{k \mid k \geq c_{i-k} \big\}$, see Theorem 2.2 in \cite{Ful}.
	Nakayama algebras with a linear quiver are in a natural bijection to Dyck paths by associating to a Nakayama algebra $A$ the Dyck path given as the top boundary of the Auslander-Reiten quiver of $A$, see the preliminaries in \cite{MRS} for full details.
	
	On the other hand, Dyck paths are in natural bijection to 321-avoiding permutations via the Billey-Jockusch-Stanley bijection that was introduced in \cite{BJS}. We will explain the Billey-Jockusch-Stanley bijection in the last section.
	Using those two bijections, we see that we can associate to every 321-avoiding permutation $\pi$ in a bijective way a Nakayama algebra $A_{\pi}$.
	\section{Translation} 
	In the following we give elementary translations of the homological notions in Theorem \ref{conjecture}.
	\begin{lemma}
		\label{extformula}
		Let~$A$ be a finite dimensional algebra with a simple $A$-module $S$.
		\begin{enumerate}  
			\item  Let $M$ be an $A$-module with minimal projective resolution
			$$\cdots P_i \rightarrow \cdots P_1 \rightarrow P_0 \rightarrow M \rightarrow 0.$$
			For $l \geq 0$, $\Ext_A^l(M,S) \neq 0$ if and only if there is a surjection $P_l \rightarrow S$.
			\item   Dually, let
			$$0 \rightarrow M \rightarrow I_0 \rightarrow I_1 \rightarrow \cdots \rightarrow I_i \rightarrow \cdots $$
			be a minimal injective coresolution of~$M$.
			For $l \geq 0$, $\Ext_A^l(S,M) \neq 0$ if and only if there is an injection $S \rightarrow I_l$.
		\end{enumerate}
	\end{lemma}
	
	\begin{proof}
		See for example~\cite[Corollary~2.6.5]{Ben}.
	\end{proof}

	\begin{lemma} \label{injdimcharalemma}
		Let $A$ be a finite dimensional algebra with $n$ simple modules.
		\begin{enumerate}
			\item For a natural number $k \geq 1$ and an indecomposable module $N$, we have \linebreak $\Ext_A^k(J,N)=0$ if and only if the injective dimension of $N$ is at most $k$.
			\item For algebras of finite global dimension, $\Ext_A^1(J,J) =0$ if and only if the algebra is hereditary, that is $\gldim A=1$.
		\end{enumerate}
	\end{lemma}
	\begin{proof}
		\phantom{aa}
		\begin{enumerate}
			\item 
			Let 
			$0 \rightarrow N \rightarrow I^0 \rightarrow \cdots \rightarrow I^i \rightarrow \cdots $
			be a minimal injective coresolution of $N$.
			We have $\Ext_A^{k}(J,N)=\Ext_A^{k}(\Omega^1(A/J),N) = \Ext_A^{k+1}(A/J,N)$. Now by \ref{extformula} (2), we see that $\Ext_A^{k+1}(A/J,N)$ is zero if and only if the term $I^{k+1}$ is zero (since $A/J$ has every simple module as a direct summand), which is equivalent to $N$ having injective dimension at most $k$.
			\item Here we use the result that the global dimension of an algebra with finite global dimension is equal to the injective dimension of its Jacobson radical, see \cite{Mar2}.
			We have $\Ext_A^1(J,J) =0$ if and only if $\Ext_A^1(J,e_i J) =0$ for all $i=1,2,\dots,n$. Now by (1) of this lemma $\Ext_A^1(J,e_i J) =0$ for all $i=1,2,\dots,n$ if and only if the injective dimension of each $e_i J$ is at most one and thus also the injective dimension of $J$ is at most one which is equivalent to $A$ having global dimension at most one.
			\qedhere	\end{enumerate}
	\end{proof}
	
	\begin{remark}
		In \cite{CIM} we show that for any finite dimensional algebra  $A$ with Jacobson radical $J$ the injective dimension of the Jacobson radical $J$ is equal to the global dimension of $A$. Using this, one can show with the same proof as in (2) of the previous lemma that $\Ext_A^1(J,J)=0$ if and only if $A$ is hereditary for general algebras $A$.
	\end{remark}
	By (2) of \ref{injdimcharalemma}, every quiver algebra $A=KQ/I$ of finite global dimension with non-zero relations $I$ has that $\Ext_A^1(J,J) \neq 0$, which motivates us to study this vector space for Nakayama algebras with a linear quiver here. Note that the number of Nakayama algebras with a linear quiver and $n$ simple modules is equal to the Catalan number $C_{n-1}$ and only one such algebra is hereditary, namely the one with Kupisch series $[n,n-1,\dots,2,1]$.
	
	We give a general statement when an indecomposable module over such an algebra has injective dimension at most one and then specialize in the next two corollaries to modules that are projective or powers of radicals of indecomposable projective modules.
	\begin{proposition} \label{translatepropo}
		Let $A$ be a Nakayama algebra. The indecomposable module $e_i A / e_i J^k$ has injective dimension at most one if and only if ($k = d_{i+k-1}$) or ( $k < d_{i+k-1}$ and $d_{i+k-1}-k=d_{i-1}$).
		$k = d_{i+k-1}$ holds if and only if $e_i A / e_i J^k$ is injective.
		
	\end{proposition} 
	\begin{proof}
		The following short exact sequence gives the injective envelope and cokernel of $e_i A/ e_i J^k$ (see for example the preliminaries in \cite{Mar}):
		$$0 \rightarrow e_i A/ e_i J^k \rightarrow D(A {e_{i+k-1}}) \rightarrow D(J^k {e_{i+k-1}}) \rightarrow 0.$$
		The injective envelope of $D(J^k {e_{i+k-1}})$ is $D(Ae_{i-1})$. 
		This shows that $e_i A/ e_i J^k$ is injective if and only if it is isomorphic to $D(A {e_{i+k-1}})$, which is equivalent to the condition that both modules have the same vector space dimension since $e_i A/ e_i J^k$ embeds into $D(A {e_{i+k-1}})$.
		Thus $e_i A/ e_i J^k$ is injective if and only if $k=\dim(e_i A/ e_i J^k)=\dim(D(A {e_{i+k-1}}))=d_{i+k-1}$. Now assume that $e_i A/ e_i J^k$ is not injective, which means that $k < d_{i+k-1}$.
		Then $e_i A/ e_i J^k$ has injective dimension equal to one if and only if $\Omega^{-1}(e_i A/ e_i J^k)=D(J^k {e_{i+k-1}})$ is injective which is equivalent to $D(J^k {e_{i+k-1}})$ having the same vector space dimension as its injective envelope $D(Ae_{i-1})$. This translates into the condition $d_{i+k-1} - k = \dim(D(J^k {e_{i+k-1}}))=\dim(D(Ae_{i-1}))=d_{i-1}$.
	\end{proof}
	
	\begin{corollary} \label{injdimprojectives}
		A module of the form $e_i A$ has injective dimension equal to one if and only if $c_i < d_{i+c_i-1}$ and $d_{i+c_i-1}-c_i=d_{i-1}$.
		
	\end{corollary}
	\begin{proof}
		Just set $k=c_i$ in Proposition \ref{translatepropo}.			
	\end{proof}
	\begin{remark}
		When $A$ has $n$ simple modules, the module $e_{n-1} J^1$ is the zero module and thus always has injective dimension zero. In fact this is the only module of the form $e_i J^1$ of injective dimension zero as those modules are proper submodules of another indecomposable module, namely $e_i A$, when they are non-zero.
		
	\end{remark}
	
	\begin{corollary} \label{corollaryradicals}
		A module of the form $e_s J^t$ for $1 \leq t \leq c_s -1$ and $s \neq n-1$ has injective dimension at most one if and only if $d_{s+c_s-1}-c_s+t=d_{s+t-1}$ and in this case the injective dimension is equal to one.
		Especially:
		$e_s J^1$ has injective dimension at most one if and only if $d_{s+c_s-1}-c_s-t=d_{s+t-1}$.

	\end{corollary}
	\begin{proof}
		Note that $e_s J^t$ for $t \geq 1$ is a proper submodule of $e_s A$ and thus is never injective since $e_s A$ is indecomposable and an injective proper submodule would show that $e_s J^t$ is a direct summand, which is absurd.
		The projective cover of $e_s J^t$ is given by $f_t : e_{s+t} A \rightarrow e_s J^t$ and by comparing dimensions we see that $\ker(f_t)= e_{s+t} J^{c_s -t}$. By the first isomorphism theorem, we have $e_s J^t \cong e_{s+t} A /e_{s+t} J^{c_s -t}$.
		Now we can use Proposition \ref{translatepropo} and set $i:=s+t$ and $k:=c_s-t$, to see that $e_s J ^t$ has injective dimension equal to one if and only if $d_{s+c_s-1}-c_s-t=d_{s+t-1}$.
	\end{proof}
	
	The previous results gave an algebraic characterization of modules of injective dimension at most one in Nakayama algebras. In the final section we will use a more pictorial description of those modules.
	The next result shows how to calculate $\dim(\Ext_A^1(J,J))$ in terms of radicals of indecomposable projective modules with injective dimension at most one.
	\begin{theorem} \label{extjjtheorem}
		Let $A$ be a Nakayama algebra with a linear quiver and Jacobson radical $J$.
		Then $\dim(\Ext_A^1(J,J))=n- | \{ e_i J | \id(e_i J) \leq 1 \}|$. 
		
	\end{theorem}
	
	\begin{proof}
		We have $\dim(\Ext_A^1(J,J))= \sum\limits_{i=0}^{n-1}{\dim(\Ext_A^1(J,e_i J))}$.
		Now $\dim(\Ext_A^1(J,e_i J))$ is non-zero if and only if $e_i J$ has injective dimension larger than one by Lemma \ref{injdimcharalemma}. In case\linebreak $\dim(\Ext_A^1(J,e_i J)) \neq 0$, we have that $\dim(\Ext_A^1(J,e_i J))=1$, since 
		$$\Ext_A^1(J,e_i J)=\Ext_A^1(\Omega^1(A/J),e_iJ)=\Ext_A^2(A/J,e_iJ)$$ counts by Lemma \ref{extformula} the number of indecomposable summands of $I^2$ when $(I^i)$ is a minimal injective coresolution of $e_i J$. But $I^2$ is indecomposable since $A$ is a Nakayama algebra and thus $\dim(\Ext_A^1(J,e_i J))=1$. 			
	\end{proof}
	
	Since we need to count the indecomposable projective modules with injective dimension one, the following is relevant:
	\begin{lemma} \label{samenumberprojdimone}
		Let $A$ be a Nakayama algebra.
		The number of indecomposable projective $A$-modules with injective dimension equal to one is equal to the number of indecomposable injective $A$-modules with projective dimension equal to one.
		
	\end{lemma}
	\begin{proof}
		Since Nakayama algebras have dominant dimension at least one, see \cite{AnFul}, we have for each indecomposable projective module $P$ of injective dimension equal to one the following short exact sequence:
		
		$$0 \rightarrow P \rightarrow I(P) \rightarrow \Omega^{-1}(P) \rightarrow 0.$$
		Here $I(P)$ is the injective envelope of $P$. Since $P$ has injective dimension one, $\Omega^{-1}(P)$ is injective and since $I(P)$ is projective-injective (using that the dominant dimension is at least one) $\Omega^{-1}(P)$ has projective dimension one.
		Thus $\Omega^{-1}(-)$ induces a bijection between the set of indecomposable projective modules of injective dimension one and indecomposable injective modules with projective dimension one with inverse $\Omega^1(-)$.
	\end{proof}
	Another consequence of Nakayama algebras having dominant dimension at least one is that we have $\id (e_i J^1) \leq 1$ if and only if $\Omega^{-1}(e_i J^1)$ is an indecomposable injective module whose first Syzygy is a radical of a projective module. We note this also as a lemma:
	\begin{lemma} \label{lemmaext1jj}
		Let $A$ be a Nakayama algebra. Then $| \{ e_i J | \id(e_i J) \leq 1 \}|$ equals the number of indecomposable injective modules $I$ such that $\Omega^1(I)$ is isomorphic to the radical of a projective module.
	\end{lemma}

	\section{Relation to 321-avoiding permutations and proofs}
	\subsection{Dyck paths}\label{dp}
	A \emph{Dyck $n$-path} $\mathcal{D}$ is a lattice path from $(0,0)$ to $(2n,0)$ consisting of $n$ number of upsteps $u=(1,1)$ and $n$ number of downsteps $d=(1,-1)$ that never dip below the axis $y=0$. A \emph{peak} vertex of $\mathcal{D}$ is a vertex preceded by a $u$ and followed by a $d$. Analogously, a \emph{valley} vertex of $\mathcal{D}$ is a vertex preceded by a $d$ and followed by a $u$. We refer to Figure \ref{fig}  for an illustration of our conventions.
	
	The \emph{level} of a point $(a,b)$ of $\mathcal{D}$ is the number $b+1$. In Figure \ref{fig}, for example, the peaks are of levels $4, 4, 6, 4$ and $6$, while the valleys are of levels $3,3$ and $1$.  
	
	The \emph{ascent sequence} $a=(a_1, a_2, \dots, a_{\ell})$ of $\mathcal{D}$ is a sequence of non-negative integers $a_i$, which correspond to the contiguous upsteps of $\mathcal{D}$. It is $a_1+\dots+a_{\ell}=n$. Similarly, we can define the \emph{descent sequence} $d=(d_1, d_2, \dots, d_{\ell})$ of $\mathcal{D}$. We write $\mathcal{D}=\prod_{i=1}^{\ell}u^{a_i}d^{d_i}$.
	
	For each $i=1,\dots, \ell-1$ we can define the partial sums $A_i:=\sum_{j=1}^ia_j$ and $D_i:=\sum_{j=1}^id_j$. The reason we omit the case $i=\ell$ is because we always have $A_{\ell}=D_{\ell}=n$.
	For the Dyck path $\mathcal{D}=u^nd^n$, these partial sums are the empty sequences.
	By the definition of a Dyck $n$-path we have:
	
	$$\begin{array}{lcl}
		&1\leq A_1<A_2<\cdots<A_{\ell-1}\leq n-1,&\smallbreak\smallbreak\\
		&1\leq D_1<D_2<\cdots<D_{\ell-1}\leq n-1,&\smallbreak\smallbreak\\
		&D_i\leq A_i, \text{ for } 1\leq i \leq \ell-1.&
	\end{array}$$

	We call the sequences $A:=(A_1,\dots,A_{\ell-1})$ and $D:=(D_1,\dots,D_{\ell-1})$  the \emph{partial ascent code} and the \emph{partial descent code} of $\mathcal{D}$, respectively. 
	We also call the pair  $(A,D)$ the \emph{partial-sum ascent-descent code} of $\mathcal{D}$.

	\begin{figure}[h]
		\centering
		\includegraphics[width=0.7\linewidth]{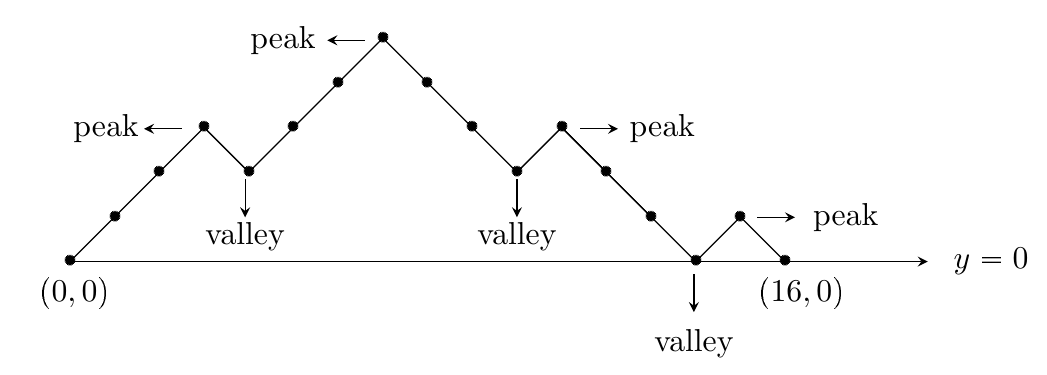}
		\caption{{An example of a Dyck $8$-path}}
		\label{fig}
	\end{figure}

	\begin{example}\label{exx} The Dyck path of Figure \ref{fig} is the Dyck $8$-path
		$\mathcal{D}=u^3d^1u^3d^3u^1d^3u^1d^1$ with ascent sequence $a=(3,3,1,1)$ and descent sequence $d=(1,3,3,1)$. Therefore, the  partial-sum ascent-descent code of $\mathcal{D}$ is $(A,D)=\big((3,6,7),\,(1,4,7)\big)$.
	\end{example}
	As mentioned in the previous section we can associate to every linear Nakayama algebra a canonical Dyck path via the top boundary of the Auslander-Reiten quiver. We assume that the reader is familiar with this construction and refer to the preliminaries of \cite{MRS} for full details. We just give one example here. 
	\begin{example}
		Let $A$ be the Nakayama algebra with the following quiver $Q$
		$$\begin{xymatrix}{ \circ^0 \ar[r]^{\alpha_1} & \circ^1 \ar[r]^{\alpha_2} & \circ^2 \ar[r]^{\alpha_3} & \circ^{3} \ar[r]^{\alpha_4} & \circ^{4}}\end{xymatrix}$$
		and relations $I=<\alpha_1 \alpha_2, \alpha_3 \alpha_4>$. Then the Kupisch series of $A$ is given by $[2,3,2,2,1]$ and the corresponding Dyck path is given by $uduuddud$. 	
	\end{example}
	Let $\mathcal{D}:=\prod_{i=1}^{\ell}u^{a_i}d^{d_i}$ a Dyck $n$-path. We define a sequence of natural numbers $k_i$, $i=1,\dots, \ell$ as follows:
	\begin{itemize}
		\item $k_1=1$.
		\item $k_i=k_{i-1}+a_{i-1}-d_{i-1}$, for all $i=2,\dots, \ell$.
	\end{itemize}
	The number $k_1$ corresponds to the level of  the point $(0,0)$ and the numbers $k_i$, $i=2,\dots, \ell$ correspond to level of the valleys of $\mathcal{D}$ (see Figure \ref{figg}).
	\begin{figure}[h]
		\centering
		\includegraphics[width=0.7\linewidth]{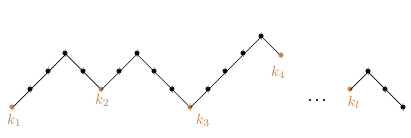}
		\caption{The points with level $k_i$}
		\label{figg}
	\end{figure}
	\begin{lemma}\label{l1} Let $\mathcal{D}:=\prod_{i=1}^{\ell}u^{a_i}d^{d_i}$ be a Dyck $n$-path and $k_i$ the natural numbers defined above. We have:
		\begin{enumerate}
			\item $k_i=1+A_{i-1}-D_{i-1}$, for all $i=2,\dots, \ell$.
			\item $k_{\ell}+a_{\ell}-d_{\ell}=1$.
		\end{enumerate}
	\end{lemma}
	\begin{proof}
		$(1)$ follows by using the definition of $k_i$ and induction on $i$. For $(2)$ we have: 
		$k_{\ell}+a_{\ell}-d_{\ell}\stackrel{(1)}{=}1+(a_1+\dots+a_{\ell})-(d_1+\dots+d_{\ell})=1+n-n=1$.
	\end{proof}
	We finish this section by reminding the reader of the basics on linear Nakayama algebras, in particular where the  projective covers, injective envelopes, and the Syzygies of indecomposable modules are located in the Dyck path corresponding to the Auslander-Reiten quiver of the algebra. This is explained in standard texts on Auslander-Reiten theory such as \cite{ARS} and in a combinatorial context in \cite{MRS}. 
	
	The indecomposable $A$-modules correspond to the lattice points with coordinates $(x_I, x_I-2\alpha)$,
	$\alpha=0,\dots,n$ in the region enclosed by the path and the $x$-axis.
	In the example of Figure \ref{na}, these are exactly the black dots. In the same example, we choose an indecomposable $A$-module $M$, by drawing a red circle around it.
	
	One can notice that each point  $(x_I, x_I-2\alpha)$ is the intersection of two diagonals; a ``right'' diagonal $R_I:\, y=x-2\alpha$ and a ``left'' diagonal $L_I:\, y=-x+2(x_I-\alpha)$. The projective cover of the corresponding module is depicted by the upper point obtained by intersecting $L_I$ and the diagram, while its injective envelope is depicted by the upper point obtained by intersecting $R_I$ and the diagram. In the example of Figure \ref{na}
	the point inside the blue circle corresponds to the projective cover of $M$,  while the point inside the green circle corresponds to its injective envelope.

	We now calculate the level of the point corresponding to the first Sygygy of an indecomposable $A$-module, by subtracting the level of the indecomposable module from the level of its projective cover. Then we depict the first Sygygy in the left diagonal the projective cover belongs to. In our example, the first Sygygy of $M$ corresponds to the point inside the  purple circle.
	
	Lastly, the radical of an indecomposable $A$-module is always one level below and in the same left diagonal as its projective cover. In our example, the radical of $M$ corresponds to the point inside the  yellow circle.

	\begin{figure}[h]
		\centering
		\includegraphics[width=1\linewidth]{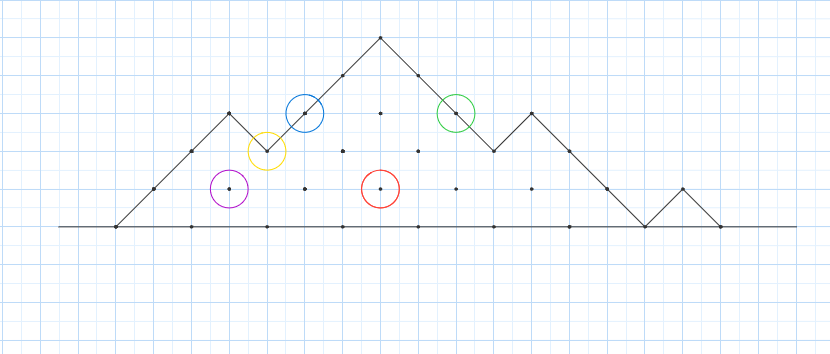}
		\caption{An example of an indecomposable $A$-module}
		\label{na}
	\end{figure}
	
	As a result of the above depiction, indecomposable projective-injective modules correspond to the peaks of the diagram and the indecomposable modules with dominant and codominant dimension at least 1 correspond to its valleys.

	\subsection{$321$-avoiding permutations} A $321$-avoiding permutation on $[n]$ is a permutation $\pi$ on $[n]$ such that there is no triple $i<j<k$ with $\pi(k)<\pi(j)<\pi(i)$.
	\begin{example}
		The permutation 	$\left(\begin{matrix}
			1 & 2 & 3 & 4&5&6& 7 & 8\\
			8&1  &5 &2 &4&3&6 &7 
		\end{matrix}\right) \in S_8$ is not a $321$-avoiding permutation, since, for example, $\pi(6)<\pi(3)<\pi(1)$.
	\end{example}
	
	\subsection{The Billey-Jockusch-Stanley bijection} This bijection \cite{BJS} is a bijection between the set of Dyck $n$-paths and the set of $321$-avoiding permutations on $\left[n\right]$ and it is described as follows: Let $\mathcal{D}$ be a Dyck path with partial-sum ascent-descent code $(A, D)$, as described in section \ref{dp}. We now obtain a partial permutation, where $A+1$ are its excedance values and $D$ its excedance locations. Filling in the missing entries in increasing order, we obtain a $321$-avoiding permutation.
	\begin{example}\label{exxx}
		Let $\mathcal{D}=u^3d^1u^3d^3u^1d^3u^1d^1$ be the Dyck $8$-path of Example \ref{exx} with partial-sum ascent-descent code $(A,D)=\big((3,6,7),\,(1,4,7)\big)$. It is $A+1=(4,7,8)$ and, hence, according to 
		the Billey-Jockusch-Stanley bijection we obtain the following partial permutation on $\left[8\right]$, with excedance values $(4,7,8)$ and excedance locations $(1,4,7)$.
		$$\left(\begin{matrix}
			1 & 2 & 3 & 4&5&6& 7 & 8\\
			4 &  & & 7&&&8 & 
		\end{matrix}\right)$$
		Filling in the missing entries in increasing order, we obtain the following $321$-avoiding permutation:
		$$\left(\begin{matrix}
			1 & 2 & 3 & 4&5&6& 7 & 8\\
			4 &	1 & 2& 7&3&5&8 & 6
		\end{matrix}\right)$$
	\end{example}
	\subsection{The inverse of the Billey-Jockusch-Stanley bijection}\label{inv} Let $\pi \in S_n$ be a 321-avoiding permutation and let $L:=\{i_1,i_2,\dots, i_{r}\}$ be the set of all excedance locations of $\pi$ (i.e. $\pi(i_k)>i_k$ for all $k=1,\dots, r$). We order the elements of $L$, such that $i_1<i_2<\dots<i_{r}$. Since $\pi$ is a 321-avoiding permutation, we have $\pi(i_1)<\pi(i_2)<\dots<\pi(i_{r})$. We now define a Dyck $n$-path as follows:
	\begin{itemize}
		\item The partial ascend code is $A:=(\pi(i_1)-1,\pi(i_2)-1,\dots,\pi(i_{r})-1)$. Hence, the ascent sequence is the following: $a=(a_1,a_2,\dots, a_{r}, a_{{r}+1})$, where $a_1=\pi(i_1)-1$, $a_{j}=\pi(i_{j})-\pi(i_{j-1})$, for all $j=2,\dots, r$ and $a_{r+1}=n+1-\pi(i_{\ell})$.
		\item 	The partial descend code is $D:=(i_1,i_2,\dots,i_{r})$. Hence, the descent sequence is the following: $d=(d_1,d_2,\dots, d_{r}, d_{r+1})$, where $d_1=i_1$, $d_{j}=i_{j}-i_{j-1}$, for all $j=2,\dots, r$ and $d_{r+1}=n-i_{r}$.
	\end{itemize}
	\begin{example}Let 
		$\pi = \left(\begin{matrix}
			1 & 2 & 3 & 4&5&6& 7 & 8\\
			4 & 1 & 2 & 7&3&5&8 & 6 
		\end{matrix}\right)
		\in S_8$ be the $321$-avoiding permutation we calculated in Example \ref{exxx}. We apply the inverse of the Billey-Jockusch-Stanley bijection. It is $r=3$, $(i_1, i_2, i_3)=(1,4,7)$ and $(\pi(i_1),\pi(i_2),\pi(i_3))=(4,7,8)$. The ascent sequence is then $a=(3,3,1,1)$ and the descent sequence is $d=(1,3,3,1)$, which correspond to the Dyck path.
	\end{example}
	\begin{lemma}\label{zz}
		Let $\pi$ be a $321$-avoiding permutation on $[n]$ with  excedance locations  $i_1<i_2<\dots<i_{r}$ and let $D=\prod_{i=1}^{r+1}u^{a_i}d^{d_i}$ be the corresponding Dyck $n$-path, obtained from the inverse of the Billey-Jockusch bijection. Then, for all $j=1,2,\dots, r$ we have:
		\begin{enumerate}
			\item $A_j=\pi(i_j)-1$.
			\item $D_j=i_j$.
		\end{enumerate}
	\end{lemma}
	\begin{proof}
		It follows directly from the definition of the inverse of the Billey-Jockusch \linebreak bijection.
	\end{proof}
	\subsection{The kernels of the projective covers of the indecomposable injective modules} 
	The next proposition counts the number of indecomposable injective modules with projective dimension one in a linear Nakayama algebra.
	\begin{proposition}\label{P}
		Let $A$ be an $(n+1)$-linear Nakayama algebra with corresponding Dyck $n$-path $\mathcal{D}$.
		The number of points $P$, which correspond to the the indecomposable injective modules with projective dimension one is:
		\begin{enumerate}
			\item $n$, if $\mathcal{D}=u^nd^n$.
			\item $d_1-1+\sum_{i=2}^{\ell-1}\max\{d_i-k_i-1,0\}+a_{\ell}-1$, if  $\mathcal{D}=\prod_{i=1}^{\ell}u^{a_i}d^{d_i}$, $\ell\geq 2$.
		\end{enumerate}
	\end{proposition}
	\begin{proof}
		\begin{enumerate}
			\item This is clear.
			\item We first consider the part $u^{a_1}d^{d_1}$. It is $a_1\geq d_1$. Since $\ell \geq 2$, the indecomposable injective modules $Q$, which are not projective, correspond to the points in the right side of the peak,  whose levels are $1+a_1-m$, where $m=1,2,\dots, d_1-1$, as we can see in the following piture:
			\begin{figure}[h]
				\centering
				\includegraphics[width=0.6\linewidth]{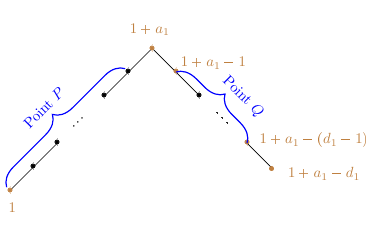}
			\end{figure}
			\\
			Hence, the points $P$ are the ones of level $1+a_1-(1+a_1-m)=m$,  $m=1,2,\dots, d_1-1$. If $d_1-1<1$, i.e. $d_1=1$, then $\#P=0$. We consider now the case, where $d_1>1$. It is $d_1-1<d_1\leq a_1<a_1+1$, hence $\#P=d_1-1$. Summarizing the two cases,  for the part  $u^{a_1}d^{d_1}$ we have $\#P=d_1-1$.
			
			We now consider the part $u^{a_{\ell}}d^{d_{\ell}}$. It is $a_{\ell}\leq d_{\ell}$. The indecomposable injective modules $Q$, which are not projective, correspond to the points in the right side of the peak,  whose levels are $k_{\ell}+a_{\ell}-m$, where $m=1,2,\dots, d_{\ell}$. Hence, the points $P$ are the ones of level $k_{\ell}+a_{\ell}-(k_{\ell}+a_{\ell}-m)=m$,  $m=1,2,\dots, d_{\ell}$.
			\vspace{1cm}
			\begin{figure}[h]
				\centering
				\includegraphics[width=0.9\linewidth]{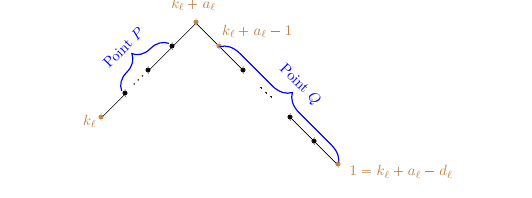}
			\end{figure}
			
			If $d_{\ell}<k_{\ell}$ then $k_{\ell}-d_{\ell}+a_{\ell}>a_{\ell}$. From Lemma \ref{l1}$(2)$ we have then that $1>a_{\ell}$, which is a contradiction. Therefore, $d_{\ell}\geq k_{\ell}$. If $d_{\ell}=k_{\ell}$ then $\#P=0$. We now consider the case $d_{\ell}> k_{\ell}$. From Lemma \ref{l1}$(2)$ we have $k_{\ell}-d_{\ell}+a_{\ell}=1>0$, hence $d_{\ell}<k_{\ell}+a_{\ell}$. Hence, 
			$\#P=d_{\ell}-k_{\ell}$. Summarizing the two cases, for the part  $u^{a_{\ell}}d^{d_{\ell}}$ we have $\#P=d_{\ell}-k_{\ell}$.	Using Lemma \ref{l1}$(2)$ again, we have $\#P=d_{\ell}-k_{\ell}=a_{\ell}-(k_{\ell}+a_{\ell}-d_{\ell})=a_{\ell}-1$.
			
			We now consider the part $u^{a_{i}}d^{d_{i}}$, where $1<i<\ell$. The indecomposable injective modules $Q$, which are not projective, correspond to the points in the right side of the peak,  whose level is $k_{i}+a_{i}-m$, where $m=1,2,\dots, d_{i}-1$. 
			\begin{figure}[h]
				\centering
				\includegraphics[width=0.6\linewidth]{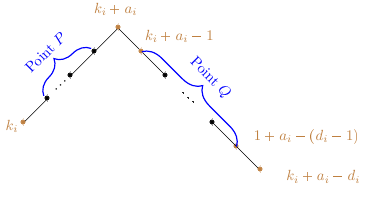}
			\end{figure}
			Hence, the points $P$ are the ones of level $k_{i}+a_{i}-(k_{i}+a_{i}-m)=m$,  $m=1,2,\dots, d_i-1$. if $d_i-1\leq k_i$ then
			$\#P=0$. We now consider the case $d_i-1> k_i$. It is $-1<k_{i+1}=k_i+a_i-d_i$. Therefore, $d_i-1<k_i+a_i$. Hence, $\#P=d_i-1-k_i$.  Summarizing the two cases, for the part  $u^{a_i}d^{d_{i}}$,  $1<i<\ell$  we have $\#P=\max \{d_i-k_i-1,0\}$.
			\qedhere	\end{enumerate}
	\end{proof}
	
	Let $\pi \in S_n$ be a 321-avoiding permutation, $\{i_1,i_2,\dots, i_r\}$ the set of all excedance locations of $\pi$, as described in Section \ref{inv}. We have the following corollary. 
	\begin{corollary} \label{corollaryP}
		Let $A$ be an $(n+1)$-linear Nakayama algebra with corresponding Dyck $n$-path $\mathcal{D}$ and $\pi$ the 321-avoiding permutation under the Billey-Jockusch-Stanley bijection applied to $\mathcal{D}$.
		The number of points $P$, which correspond to the the indecomposable injective modules with projective dimension one is:
		\begin{enumerate}
			\item $n$, if $\pi$ is the identity.
			\item $i_1-1+\sum_{j=2}^{r}\max\{i_j-\pi(i_{j-1})-1,0\}+n-\pi(i_r)$, if $\pi$ is not the identity.
		\end{enumerate}
	\end{corollary}
	\begin{proof}
		If $\pi$ is the identity, then following section \ref{inv} the corresponding $n$-Dyck path is $\mathcal{D}=u^nd^n$. Hence, $(1)$ follows directly from Proposition \ref{P} (1).
		
		Let $\pi$ now not be the identity. Following section \ref{inv} the corresponding $n$-Dyck path is $\mathcal{D}=\prod_{j=1}^{r+1}u^{a_j}d^{d_j}$, where $a_1=\pi(i_1)-1$, $a_j=\pi(i_j)-\pi(i_{j-1})$, for $j=2,\dots,r$, $a_{r+1}=n+1-\pi(i_r)$, $d_1=i_1$, $d_j=i_j-i_{j-1}$, for $j=2,\dots,r$, $d_{r+1}=n-i_r$. Following Proposition \ref{P}(2) we have that $\#P=i_1-1+\sum_{j=2}^{r}\max\{i_j-i_{j-1}-k_j-1,0\}+n-\pi(i_r)$. It remains to prove that for every $j=2,\dots, r$, $i_{j-1}+k_j=\pi(i_{j-1})$. This result follows directly from Lemma \ref{zz} and Lemma \ref{l1}.
	\end{proof}
	\begin{theorem} \label{theoremP}
		Let $\pi\in S_n$ be a 321-avoiding permutation  with $\{i_1,\dots, i_r\}$ the set of all excedance locations of $\pi$, with $i_1<i_2<\dots<i_r$. 
		The number of fixed points of $\pi$ is the following:
		\begin{enumerate}
			\item $n$, if $\pi$ is the identity.
			\item $i_1-1+\sum_{j=2}^{r}\max\{i_j-\pi(i_{j-1})-1,0\}+n-\pi(i_r)$, if $\pi$ is not the identity.
		\end{enumerate}
	\end{theorem}
	\begin{proof}
		(1) is obvious, therefore we prove (2). Let 
		$$\pi = \left(\begin{tabular}{cccc|ccc|ccc}
			1 & 2 & \ldots & $i_1-1$& $i_1$ &\ldots & $i_r$ & $i_{r}+1$&\ldots& $n$\\
			$*$ & $*$&$*$  & $*$& $\pi(i_1)$ & $*$& $\pi(i_r)$ & $*$&$*$& $*$
		\end{tabular}\right)$$
		Since $\pi$ is a $321$-avoiding permutation, we must fill the missing entries in increasing order. We have $i_1<\pi(i_1)$ and, hence, we first fill the entries with the numbers $1, 2,\dots, i_1-1$. Therefore, the permutation takes the following form:
		$$\pi = \left(\begin{tabular}{cccc|ccc|ccc}
			1 & 2 & \ldots & $i_1-1$& $i_1$ &\ldots & $i_r$ & $i_{r}+1$&\ldots& $n$\\
			$1$ & $2$& \ldots & $i_1-1$& $\pi(i_1)$&$*$ & $\pi(i_r)$ & $*$&$*$& $*$
		\end{tabular}\right)$$
		Therefore, we have at this point $i_1-1$ fixed points.
		We now consider the indices $i_r+1, i_r+2, \dots, n$. Firstly, we notice that in this  last part of the permutation we have to fill $n-i_r$ entries. Since $i_r<\pi(i_r)$,  among the numbers, with which we fill these entries in increasing order, are the numbers (in decreasing order) $n, n-1, \dots,\pi(i_r)+1$. The only possible choice is the following: 
		$$\pi = \left(\begin{tabular}{ccccccccc}
			\ldots& $i_r$ &$i_{r}+1$&$\pi(i_r)-1$ &$\pi(i_r)$&$\pi(i_r)+1$&\ldots& $n-1$& $n$\\
			&$\pi(i_r)$&$*$&$*$&$*$&$\pi(i_r)+1$&\ldots& $n-1$& $n$
		\end{tabular}\right)$$
		Therefore, we have from this last part of the permutation $n-\pi(i_r)$ fixed points. 
		
		We now consider the following part of the permutation:
		$$\pi = \left(\begin{tabular}{ccccc}
			\ldots&$i_{j-1}$ & \ldots &$i_{j}$&\ldots\\
			& $\pi(i_{j-1})$ &$*$&$\pi(i_{j})$&
		\end{tabular}\right),$$
		where $j=2,\dots, r$. 
		We distinguish the following cases:
		\begin{itemize}
			\item $\pi(i_{j-1})<i_j.$ In this case, the permutation is of the following form:
			{\scriptsize{
					$$\pi = \left(\begin{tabular}{ccccccc|ccc|cc}
						1& \ldots & $i_1$ &\ldots& $i_{j-1}$& \ldots&$\pi(i_{j-1})$&$\pi(i_{j-1})+1$&\ldots
						&$i_{j}-1$&$i_j$&\ldots\\
						$*$& $*$ & $\pi(i_1)$ &$*$& $\pi(i_{j-1})$& $*$&$*$&$*$&$*$
						&$*$&$\pi(i_j)$&$*$
					\end{tabular}\right)$$
			}}
			
			Let $x\in\{1,2,\dots, i_{j}-1\}$.  We recall that we fill in the missing entries in increasing order and that  $\pi(i_1)<\pi(i_2)<\dots<\pi(i_{j-1})<i_j$. Therefore, $\pi(x)\in \{1,2,\dots, i_{j}-1\}$.
			In particular, we fill the entries 	$\pi(x)$, $x\in\{1,2,\dots, \pi(i_{j-1})\}\setminus\{i_1,\dots, i_{j-1}\}$ in increasing order with the numbers belonging to the set  $$\{1,2,\dots, \pi(i_{j-1})\}\setminus\{\pi(i_1),\dots, \pi(i_{j-1})\}.$$ Therefore, we have $\pi(x)\in \{1,2,\dots, \pi(i_{j-1})\}$, for $x\in \{1,2,\dots, \pi(i_{j-1})\}$. 
			
			Let now $x\in\{\pi(i_{j-1})+1,\pi(i_{j-1})+2,\dots,i_j-1\}$. We fill the entries $\pi(x)$ in increasing order with the numbers belonging to the set $\{\pi(i_{j-1})+1,\pi(i_{j-1})+2,\dots,i_j-1\}$. Therefore, we have $\pi(x)=x$, for $x\in\{\pi(i_{j-1})+1,\pi(i_{j-1})+2,\dots,i_j-1\}$ and hence, in this part of the permutation we have $i_j-1-\pi(i_{j-1})$ fixed points.
			\item $\pi(i_{j-1})\geq i_j.$ 
			Let $S_j:=\{\ell\in\{1,\dots, i_{j-1}\}\,:\,\pi(\ell)\geq i_j\}$. It is $S_j\not=\emptyset$, since $i_{j-1}\in S_j$. Let $x\in\{1,\dots, i_j-1\}\setminus \{\ell\;:\; \ell \in S_j\}$. We have $\pi(x)\in\{1,\dots, i_j-1-|S_j|\}$. 
			In particular,  for the elements $x\in\{i_{j-1}+1,\dots, i_j-1\}$ we have $\pi(x)\leq x-|S_j|$ and, hence, there are no fix points in this part of the permutation.
			
		\end{itemize}
		
		Combining the two cases, the number of fixed points in this part of the permutation is $\max\{i_j-\pi(i_{j-1})-1,0\}$.
	\end{proof}
	
	We can now give a proof of our first main result:
	\begin{theorem}
		Let $A_{\pi}$ be a Nakayama algebra corresponding to the 321-avoiding permutation $\pi$.
		Then the number of indecomposable projective $A$-modules with injective dimension one is equal to the number of fixed points of $\pi$.
		
	\end{theorem}
	\begin{proof}
		By Lemma \ref{samenumberprojdimone} the number of indecomposable projective $A$-modules with injective dimension one equals the number of indecomposable injective $A$-modules with projective dimension one. Those modules were counted for a general Nakayama algebra with corresponding Dyck path $\mathcal{D}$ in Corollary \ref{corollaryP} and in Theorem \ref{theoremP} we saw that this number coincides with the fixed points of the 321-avoiding permutation $\pi$ which is the image of $D$ under the Billey-Jockusch-Stanley bijection.		
	\end{proof}

	\subsection{The number of self-extensions of the Jacobson radical of a Nakayama algebra}
	Let $A$ be an $(n+1)$-linear Nakayama algebra. 
	In Theorem \ref{extjjtheorem} we saw that $\dim(\Ext_A^1(J,J))=n+1- | \{ e_i J | \id(e_i J) \leq 1 \}|$ and in Lemma \ref{lemmaext1jj} that $| \{ e_i J | \id(e_i J) \leq 1 \}|$ equals the number of indecomposable injective modules $I$ whose first syzygy is the radical of a projective module. We will now count those modules for a general $(n+1)$-linear Nakayama algebra $A$ corresponding to a Dyck $n$-path $\mathcal{D}$.
	\begin{proposition}\label{PP}
		Let $A$ be an $(n+1)$-linear Nakayama algebra with corresponding Dyck $n$-path $\mathcal{D}=\prod_{i=1}^{\ell}u^{a_i}d^{d_i}$, $\ell\geq 2$.
		The number of points $P$, which correspond to the indecomposable injective modules whose first Syzygy is a radical of a projective module is:
		$$d_1+\sum_{i=2}^{\ell-1}\max\{d_i-k_i,0\}+a_{\ell}.$$
	\end{proposition}
	\begin{proof}
		
		We first consider the part $u^{a_1}d^{d_1}$. It is $a_1\geq d_1$. The injective modules $Q$
		correspond to the points of the right side of the peak, whose level is $1+a_1-m$, where $m=1,2,\dots, d_1-1$.
		\begin{figure}[h]
			\centering
			\includegraphics[width=0.6\linewidth]{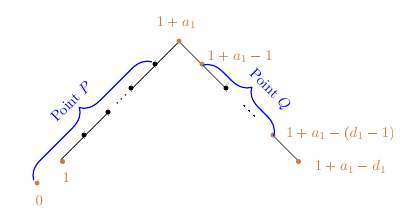}
		\end{figure}
		\\	Their Syzygies are $1+a_1-(1+a_1-m)=m$, $m=1,2,\dots, d_1-1$. 
		We notice that in this case, all these Syzygies correspond to points $P$ (since $d_1\leq a_1)$ and, hence, including also the 0, we have $\#P=d_1$.
		
		We now consider the part $u^{a_{\ell}}d^{d_{\ell}}$. It is $a_{\ell}\leq d_{\ell}$. The points $Q$, which are the injective modules, correspond to the points of the right side of the peak,  whose level is $k_{\ell}+a_{\ell}-m$, where $m=1,2,\dots, d_{\ell}$. 
		\begin{figure}[h]
			\centering
			\includegraphics[width=0.74\linewidth]{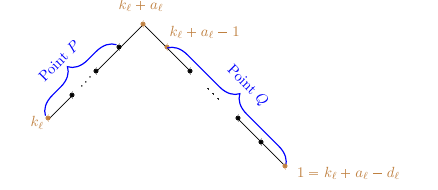}
		\end{figure}
		
		The Syzygy of the injective modules correspond to the points  $k_{\ell}+a_{\ell}-(k_{\ell}+a_{\ell}-m)=m$,  $m=1,2,\dots, d_{\ell}$.
		Hence, the level of the points we are interested is the intersection of the intervals $\left[1,2,\dots, d_{\ell}\right]$ and $\left[k_{\ell},k_{\ell}+1, \dots, k_{\ell}+a_{\ell}-1 \right]$. From Lemma \ref{l1}$(2)$ we have $k_{\ell}+a_{\ell}-1=d_{\ell}$, therefore the intersection of the above intervals is the interval $\left[k_{\ell},k_{\ell}+1, \dots, k_{\ell}+a_{\ell}-1 \right]$. Therefore, 
		$\#P=a_{\ell}$.
		
		We now consider the part $u^{a_{i}}d^{d_{i}}$, where $1<i<\ell$. The injective modules correspond to points $Q$, whose level is $k_{i}+a_{i}-m$, where $m=1,2,\dots, d_{i}-1$. 
		\begin{figure}[h]
			\centering
			\includegraphics[width=0.74\linewidth]{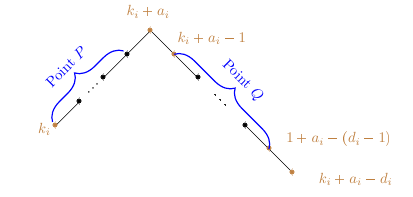}
		\end{figure}
		\\
		The Syzygies of the injective modules are those of level $k_{i}+a_{i}-(k_{i}+a_{i}-m)=m$,  $m=1,2,\dots, d_i-1$.
		Hence, the levels of the points $P$ is the intersection of the intervals $\left[1,2,\dots, d_i-1\right]$ and $\left[k_{i},k_{i}+1, \dots, k_{i}+a_{i} \right]$. We first notice that $k_i+a_i> d_i-1$, since by definition $k_i+a_i-d_i=k_{i+1}$. If $d_i\leq k_i$ then
		$\#P=0$. If $d_i> k_i$ then the intersection of the above intervals is the interval  $\left[k_{i},k_{i}+1, \dots, d_i-1 \right]$. Hence, $\#P=d_i-k_i$.
		Summarizing the two cases, for the part  $u^{a_i}d^{d_{i}}$,  $1<i<\ell$  we have $\#P=\max \{d_i-k_i,0\}$.
	\end{proof}
	
	Let $\pi \in S_n$ be a 321-avoiding permutation, $\{i_1,i_2,\dots, i_r\}$ the set of all excedance locations of $\pi$, as described in Section \ref{inv}. We assume that $\pi$ is not the identity.  We have the following corollary. 
	\begin{corollary} \label{corollaryPP}
		Let $A$ be an $(n+1)$-linear Nakayama algebra with corresponding $n$-Dyck path $\mathcal{D}$ and let $\pi$ be the image of $\mathcal{D}$ under the Billey-Jockusch-Stanley bijection.
		The number of points $P$, which correspond to the indecomposable injective modules whose first Syzygy is a radical of a projective module is:
		$$i_1+\sum_{j=2}^{r}\max\{i_j-\pi(i_{j-1}),0\}+n+1-\pi(i_r).$$
	\end{corollary}
	\begin{proof}
		Following Section \ref{inv} the corresponding $n$-Dyck path is $\mathcal{D}=\prod_{j=1}^{r+1}u^{a_j}d^{d_j}$, where $a_1=\pi(i_1)-1$, $a_j=\pi(i_j)-\pi(i_{j-1})$, for $j=2,\dots,r$, $a_{r+1}=n+1-\pi(i_r)$, $d_1=i_1$, $d_j=i_j-i_{j-1}$, for $j=2,\dots,r$, $d_{r+1}=n-i_r$. Following Proposition \ref{PP} we have that $\#P=i_1+\sum_{j=2}^{r}\max\{i_j-i_{j-1}-k_j,0\}+n+1-\pi(i_r)$. It remains to prove that $i_{j-1}+k_j=\pi(i_{j-1})$, for $j=2,\dots, r$. This result follows directly from Lemma \ref{zz} and Lemma \ref{l1}.
	\end{proof}
	
	\begin{definition}
		The \emph{connectivity set} of a permutation $\sigma\in S_n$ is the set of indices $1\leq i \leq n$ such that $\sigma(k)<i$ for all $k<i$. The \emph{support size} $\mathfrak{s}(\sigma)$ of $\sigma$ is defined as the cardinality $k$ such that $\pi$ is the minimal product of transpositions of the form $s_i$ and $k$ is the number of distinct $s_i$ that appear.
	\end{definition}
	
	We refer to \url{http://www.findstat.org/StatisticsDatabase/St000019} for more data on the support size. In particular we have that the connectivity set is the complement of the support. And thus we can calculate the support size of a permutation using the connectivity set as we will do in the following.
	\begin{theorem} \label{theoremPP}
		Let $\pi\in S_n$ be a 321-avoiding permutation (not the identity)  with $\{i_1,\dots, i_r\}$ the set of all excedance locations of $\pi$, corresponding to an $n$-Dyck path using the Billey-Jockusch-Stanley bijection. 
		The connectivity set of $\pi$ has the following cardinality:
		$$i_1+\sum_{j=2}^{r}\max\{i_j-\pi(i_{j-1}),0\}+n-\pi(i_r).$$
	\end{theorem}
	\begin{proof}
		Let
		$$\pi = \left(\begin{tabular}{cccccccccc}
			1 & 2 & \ldots & $i_1-1$& $i_1$ &\ldots & $i_r$ & $i_{r}+1$&\ldots& $n$\\
			$*$ & $*$&$*$  & $*$& $\pi(i_1)$ & $*$& $\pi(i_r)$ & $*$&$*$& $*$
		\end{tabular}\right)$$
		According to the Billey-Jockusch-Stanley bijection, we must fill the missing entries in increasing order. Since $i_1<\pi(i_1)$ we first fill the entries with the numbers $1, 2,\dots, i_1-1$. Therefore, the permutation takes the following form:
		$$\pi = \left(\begin{tabular}{cccccc}
			1 & 2 & \ldots & $i_1-1$& $i_1$ &\ldots \\
			$1$ & $2$& \ldots & $i_1-1$& $\pi(i_1)$&$*$ 
		\end{tabular}\right)$$
		Since $i_1<\pi(i_1)$ the points $1,2, \dots, i_1$ belong to the connectivity set. 
		Therefore, we have at this point $i_1$ points inside the connectivity set.
		
		We now consider the indices $i_r+1, i_r+2, \dots, n$. Firstly, we notice that in this  last part of the permutation we have to fill $n-i_r$ entries. Since $i_r<\pi(i_r)$,  among the numbers, with which we fill these entries in increasing order, are the numbers (in decreasing order) $n, n-1, \dots,\pi(i_r)+1$. The only possible choice is the following: 
		$$\pi = \left(\begin{tabular}{cc|ccccccc}
			\ldots& $i_r$ &$i_{r}+1$&$\pi(i_r)-1$ &$\pi(i_r)$&$\pi(i_r)+1$&\ldots& $n-1$& $n$\\
			&$\pi(i_r)$&$*$&$*$&$*$&$\pi(i_r)+1$&\ldots& $n-1$& $n$
		\end{tabular}\right)$$
		Therefore, for each $i\in\{i_r+1, \dots, \pi(i_r)\}$ we have $i\leq \pi(i_r)$ and, hence, these points don't belong to the connectivity set of the permutation. On the other hand, the indices $\pi(i_r)+1, \dots, n-1, n$ belong to this set. Hence, we have $n-\pi(i_r)$ elements inside the connectivity set. 
		
		We now consider the indices between $i_{j-1}+1$ and $i_j$, for $j=2,\dots, r$. We distinguish the following cases:
		\begin{itemize}
			\item $\pi(i_{j-1})<i_j.$ In this case, the permutation is of the following form:
			{\scriptsize{
					$$\pi = \left(\begin{tabular}{ccccccc|ccc|cc}
						1& \ldots & $i_1$ &\ldots& $i_{j-1}$& \ldots&$\pi(i_{j-1})$&$\pi(i_{j-1})+1$&\ldots
						&$i_{j}-1$&$i_j$&\ldots\\
						$*$& $*$ & $\pi(i_1)$ &$*$& $\pi(i_{j-1})$& $*$&$*$&$*$&$*$
						&$*$&$\pi(i_j)$&$*$
					\end{tabular}\right)$$
			}}
			
			Let $x\in\{1,2,\dots, i_{j}-1\}$.  By the definition of the Billey-Jockusch-Stanley bijection and the fact that  $\pi(i_1)<\pi(i_2)<\dots<\pi(i_{j-1})<i_j$, we have $\pi(x)\in \{1,2,\dots, i_{j}-\nolinebreak1\}$.
			In particular, we fill the entries 	$\pi(x)$, $x\in\{1,2,\dots, \pi(i_{j-1})\}\setminus\{i_1,\dots, i_{j-1}\}$ in increasing order with the numbers belonging to the set $$\{1,2,\dots, \pi(i_{j-1})\}\setminus\{\pi(i_1),\dots, \pi(i_{j-1})\}.$$ Therefore, we have $\pi(x)\in \{1,2,\dots, \pi(i_{j-1})\}$, for $x\in \{1,2,\dots, \pi(i_{j-1})\}$. 
			
			Let now $x\in\{\pi(i_{j-1})+1,\pi(i_{j-1})+2,\dots,i_j-1\}$. We fill the entries $\pi(x)$ in increasing order with the numbers belonging to the set $\{\pi(i_{j-1})+1,\pi(i_{j-1})+2,\dots,i_j-1\}$. Therefore, we have $\pi(x)=x$, for $x\in\{\pi(i_{j-1})+1,\pi(i_{j-1})+2,\dots,i_j-1\}$.
			Therefore, for each $x\in\{i_{j-1}+1, \dots, \pi(i_{j-1})\}$ we have $x\leq \pi(i_{j-1})$ and, hence, these points don't belong to the connectivity set of the permutation. On the other hand, the indices $\pi(i_{j-1})+1, \dots, i_j$ belong to this set. Hence, we have $i_j-\pi(i_{j-1})$ elements inside the connectivity set. 
			
			\item $\pi(i_{j-1})\geq i_j.$ Firstly, we notice that since $\pi(i_{j-1})\geq i_j$, the index $i_j$ doesn't belong to the connectivity set. We now consider the indices $x\in \{i_{j-1}+1,\dots, i_j-1\}$.
			Let $S_j:=\{\ell\in\{1,\dots, i_{j-1}\}\,:\,\pi(\ell)\geq i_j\}\subset \{i_1,\dots,i_{j-1}\}$. We have $S_j\not=\emptyset$, since $i_{j-1}\in S_j$. Let $x\in\{1,\dots, i_j-1\}\setminus \{\ell\;:\; \ell \in S_j\}$. We have $\pi(x)\in\{1,\dots, i_j-1-|S_j|\}$. 
			In particular, by the definition of Billey-Jockusch-Stanley bijection, for the elements $x\in\{i_{j-1}+1,\dots, i_j-1\}$ we have $\pi(x)\leq x-|S_j|$ and, hence, these indices don't  belong to connectivity set. 
		\end{itemize}
		\noindent
		Therefore, by combining the two cases, we have that the number of points in the connectivity set is $\max\{i_j-\pi(i_{j-1}),0\}$.
	\end{proof}
	We now obtain our second main result:
	\begin{theorem} 
		Let $A_{\pi}$ be an $(n+1)$-linear Nakayama algebra with Jacobson radical $J$ associated to the 321-avoiding permutation $\pi$ on $[n]$.
		Then $\Ext_{A_{\pi}}^1(J,J) \cong K^{\mathfrak{s}(\pi)}$.
	\end{theorem} 
	\begin{proof}
		In Theorem \ref{extjjtheorem} we saw that $\dim(\Ext_A^1(J,J))=n+1- | \{ e_i J | \id(e_i J) \leq 1 \}|$ and in Lemma \ref{lemmaext1jj} that $| \{ e_i J | \id(e_i J) \leq 1 \}|$ equals the number of indecomposable injective modules $I$ whose first Syzygy is the radical of a projective module. We counted those modules $I$ for a general  Nakayama algebra $A$ corresponding to a Dyck $n$-path $\mathcal{D}$ in Corollary \ref{corollaryPP} and in Theorem \ref{theoremPP} we saw that it coincides with the cardinality of the connectivity set of the corresponding 321-avoiding permutation $\pi$.
		Thus $\dim(\Ext_A^1(J,J))=n+1-| \{ e_i J | \id(e_i J) \leq 1 \}|$ equals the support size of $\pi$.		
	\end{proof}
	
	As a corollary of the previous theorem we obtain:
	\begin{corollary}
		Let $k$ be a natural number with $0 \leq k \leq n$.
		The number of linear Nakayama algebras $A$ and Jacobson radical $J$ with $n+2$ simple modules such that $\dim(\Ext_A^1(J,J))=k$ is equal to the number of standard tableaux of shape $[n,k]$. In particular, the number of such Nakayama algebras such that $\dim(\Ext_A^1(J,J))$ is equal to the maximal possible number $n$ is given by the Catalan numbers $C_n$.
		
	\end{corollary}
	\begin{proof}
		Let $\pi$ be a general permutation in $S_n$. Then we have that the support size $\mathfrak{s}(\pi)$ is given by $n-|\{1 \leq k \leq n \mid \{\pi_1,\dots,\pi_k\}=\{1,\dots,k\} \}|$, where $|\{1 \leq k \leq n \mid \{\pi_1,\dots,\pi_k\}=\{1,\dots,k\} \}|$ is called the \emph{block number} of $\pi$, see \url{http://www.findstat.org/StatisticsDatabase/St000056}.
		The result now follows from Proposition 2.5 of \cite{ABR}.
	\end{proof}
	\section{Outlook on generalizations}
	
	Motivated by the main result of our article, we pose the following problem:
	\begin{problem}
		Let $k \geq 0$ and $n \geq 2$.
		Give a combinatorial interpretation of the statistic that associates the number of indecomposable projective $A$-modules of injective dimension $\leq k$ to a linear Nakayama algebra with $n$ simple modules.
	\end{problem}
	In this article we gave a solution to this problem for $k=1$. For $k=0$ the problem is about the number of indecomposable projective-injective modules and those clearly correspond to the number of peaks of the Dyck path.
	We give the following combinatorial conjecture for the case $k=2$:
	\begin{conjecture}
		The statistic that associates to a Dyck path from $(0,0)$ to $(0,2n-2)$ (in canonical bijection to a linear Nakayama algebra with $n$ simple modules) the number of indecomposable projective modules with injective dimension $\leq 2$ has the same distribution as the statistic that associates the pyramid weight plus one to an $(n-1)$-Dyck path.
		
	\end{conjecture}
	Here the pyramid weight of a Dyck path $D$ is defined as the sum of the lengths of the maximal pyramids (maximal sequences of the form $u^h d^h$
	) in the path, see \url{https://www.findstat.org/StatisticsDatabase/St000144/} for this statistic in findstat and \cite{B} and \cite{DS} for references where this statistic was studied.
	We came up this this conjecture thanks to findstat and we verified this conjecture for $n \leq 10$ with the computer. We remark that findstat suggests that the pyramid statistic is related to several other statistics on pattern avoiding permutations.

	\subsection*{Acknowledgements}
	This project started as a part of the working group organized by Steffen K\"onig at the University of Stuttgart. We profited from the use of findstat \cite{Find}, the GAP-package QPA \cite{QPA} and Sage \cite{Sage}. Rene Marczinzik has been supported by the DFG with the project number 428999796. We are thankful to Brian Hopkins for referring us to the reference \cite{ABR} and to Emily Norton for proofreading our paper.

\end{document}